\pdfoutput=1
\documentclass[reqno]{amsart}
\usepackage[utf8]{inputenc}
\usepackage{bbm, amssymb, mathtools, url, enumerate}
\usepackage{amsmath}
\usepackage{tikz}
\usepackage{tikz-cd}
\usepackage[scr]{rsfso}
\usepackage{bm}
\usepackage[margin=3cm]{geometry}
\usetikzlibrary{matrix,arrows,decorations.pathmorphing, decorations.markings, cd}
\usepackage{stmaryrd, rotate}
\usepackage[alphabetic]{amsrefs}
\usepackage{microtype}
\usepackage[unicode=true,pdfusetitle,
 bookmarks=true,bookmarksnumbered=false,bookmarksopen=false,
 breaklinks=false,pdfborder={0 0 0},pdfborderstyle={},backref=false,colorlinks=true]
 {hyperref}
\hypersetup{
    colorlinks, linkcolor=magenta,
    citecolor=cyan, urlcolor=magenta
}

\usepackage[capitalise]{cleveref}

\pgfdeclarelayer{bg}
\pgfsetlayers{bg,main}
\usetikzlibrary{calc}

\newtheorem{introtheorem}{Theorem}
\newtheorem{introcorollary}[introtheorem]{Corollary}

\newtheorem{theorem}{Theorem}[section]
\newtheorem{corollary}[theorem]{Corollary}
\newtheorem{lemma}[theorem]{Lemma}
\newtheorem{proposition}[theorem]{Proposition}

\theoremstyle{definition}

\newtheorem{convention}[theorem]{Convention}
\newtheorem{definition}[theorem]{Definition}

\newtheorem{example}[theorem]{Example}

\newtheorem{observation}[theorem]{Observation}
\newtheorem{remark}[theorem]{Remark}
\newtheorem*{introremark}{Remark}
\newtheorem*{introexample}{Example}

\newcommand{\biglrcorner}{\mathbin{\scalebox{1.6}{$\lrcorner$}}}

\newcommand{\notehelper}[3]{\textcolor{#3}{$\blacksquare$}\marginpar{\ifodd\thepage\raggedright\else\raggedleft\fi\color{#3}\tiny \textbf{#2:} #1}}

\newcommand{\lperp}[1]{{}^\perp#1}
\newcommand{\APh}{\lperp{\Ph}}
\newcommand{\Cell}{\mathsf{Cell}}

\usepackage{macros}

\title{Cocompactness and Presentability}
\author{Thorger Geiß}
\author{Phil Pützstück}
\author{Maxime Ramzi}
\hypersetup{
  pdfauthor={Thorger Geiß, Phil Pützstück, Maxime Ramzi}
}

\begin{document}
\begin{abstract}
    We give a short proof that $\kappa$-cocompact objects in a presentable category are subterminal.
    As our main result, we extend this to the setting of presentable $\infty$-categories.
    A consequence is that an $\infty$-category $\cC$ such that both $\cC$ and $\cC^\op$ are presentable is a small complete lattice,
    extending a classical theorem of Gabriel--Ulmer.
    Along the way, we prove a nilpotence result for phantom maps in general pointed presentable $\infty$-categories.
    Additionally, we show that a strengthening of our main result
    is equivalent to the existence of a proper class of measurable cardinals.
\end{abstract}

\begingroup\parskip=0pt
\maketitle
\tableofcontents
\endgroup

\section*{Introduction}

It is a classical theorem of category theory that if $\cC$ is a category
such that both $\cC$ and $\cC^\op$ are presentable,
then $\cC$ is a small complete lattice
(\cite{Gabriel-Ulmer}*{Satz 7.13}, \cite{Adamek-Rosicky}*{Theorem 1.64}).
In the setting of homotopy-coherent mathematics, the analogous statement for $\infty$-categories
appears to exist only as unverified folklore (see e.g.~\cite{Gestalten}*{Remark 1.25, discussion after Theorem 7.1} for claims in this vein).
We point out that it does not seem feasible to adapt the classical proof to this setting as it crucially relies on the prevalence of monomorphisms
(specifically, that morphisms admitting retracts are monomorphisms) in ordinary category theory,
which completely fails in the setting of $\infty$-categories.

Instead, we conduct a more in-depth study of ($\kappa$-)cocompact objects in presentable $\infty$-categories,
which culminates in the following significantly stronger result.

\begin{introtheorem}[Theorem \ref{thm:main}]\label{introthm:main}
    Let $\cC$ be a presentable $\infty$-category and $\kappa$ a regular cardinal.
    Then every $\kappa$-cocompact object in $\cC$ is $(-1)$-truncated.
\end{introtheorem}

\begin{introremark}
    To the best of our knowledge,
    this is a new result even for ordinary categories $\cC$.
\end{introremark}

\begin{introexample}
    Let $\CHaus$ denote the category of compact Hausdorff spaces.
    Recall that $\CHaus^\op$ is $\aleph_1$-presentable with generator $[0,1]$
    (this is essentially a consequence of Urysohn's Lemma
    and well-known to experts, see e.g.~\cite[Proposition F.5]{Efimov}).
    In particular, by Theorem \ref{introthm:main}
    and the fact that any non-empty space admits at least two maps
    to $2 = * \amalg * \in \CHaus$,
    it follows that the empty space $\emptyset$ is the only object in $\CHaus$
    that is $\kappa$-compact for any regular cardinal $\kappa$.
\end{introexample}

\begin{introcorollary}[Corollary \ref{cor:bipresentable}]\label{introcor:bipresentable}
    Let $\cC$ be an $\infty$-category such that both $\cC$ and $\cC^\op$
    are presentable. Then $\cC$ is a small complete lattice.
\end{introcorollary}

The deduction of Corollary \ref{introcor:bipresentable}
from Theorem \ref{introthm:main}
is elementary and explained in Section \ref{sec:posets}.
In the case of ordinary categories,
the proof of Theorem \ref{introthm:main} is surprisingly simple,
and we give a standalone treatment in Section \ref{sec:1-cat-case}.
The strategy for the general $\infty$-categorical case is similar
but requires more work.
We first reduce to the case of pointed $\infty$-categories in
Section \ref{sec:red-to-pointed}.
The key ingredient in this case is a result on phantom maps
we prove as Proposition \ref{prop:phantom-nilpotent}.
Specifically, we show that in any pointed $\kappa$-presentable $\infty$-category,
the ideal of maps which can be written as $n$-fold composites of $\kappa$-phantom
maps for all $n$ squares to zero.
In Section \ref{sec:tails}, we construct a sufficient supply of such phantom maps, which lets us leverage the aforementioned result to prove Theorem \ref{introthm:main}.

Finally, in Section \ref{sec:measurable},
we highlight an interesting relation of the above theorem
to large cardinal principles of set theory.
Specifically, we show that a stronger version of Theorem \ref{introthm:main},
in which $\kappa$-cocompactness is replaced with the notion of $\kappa$-cosumpactness (see Definition \ref{def:cosumpact}),
is equivalent to the existence of a proper class of measurable cardinals.

\begin{introtheorem}[Theorem \ref{thm:meas-card-eqv}]\label{introthm:measurable}
    The following statements are equivalent:
    \begin{enumerate}
        \item There exists a proper class of measurable cardinals.

        \item (Strengthened Theorem \ref{introthm:main})
            For any presentable $\infty$-category $\cC$ and regular cardinal $\kappa$,
            every $\kappa$-cosumpact object in $\cC$ is $(-1)$-truncated.

        \item For any presentable $\infty$-category $\cC$,
            every cosumpact object in $\cC$ is $(-1)$-truncated.

        \item For any infinite field $K$, the object $K \in \CAlg(K)$ is not cosumpact.
            (Here $\CAlg(K) = \CRing_{K/}$ is the ordinary category of $K$-algebras.)
    \end{enumerate}
\end{introtheorem}

\begin{introremark}
    Nonetheless, we can prove the conclusion of Theorem \ref{introthm:measurable}(2)
    for infinitary distributive $\infty$-categories without assuming the existence of measurable cardinals, see Proposition \ref{prop:distr-case}.
\end{introremark}

\begin{introremark}
    The entirety of Section \ref{sec:measurable}, including Theorem \ref{introthm:measurable} as well as Proposition \ref{prop:distr-case}
    mentioned in the previous remark,
    parses verbatim in the setting of ordinary categories.
\end{introremark}

\subsection*{Relation to previous work}

\begin{itemize}
    \item For ordinary categories, the conclusion of Corollary \ref{introcor:bipresentable} is classical,
        see \cite{Gabriel-Ulmer}*{Satz 7.13} or \cite{Adamek-Rosicky}*{Theorem 1.64}.

    \item For compactly generated stable $\infty$-categories
        the conclusion of Corollary \ref{introcor:bipresentable}
        can be deduced from work of Neeman (\cite[Appendix E.1]{Neeman}),
        see Remark \ref{rem:stable-cosumpact} for a more detailed discussion.

    \item The nilpotence result for phantom maps
        in pointed presentable $\infty$-categories
        (Proposition \ref{prop:phantom-nilpotent})
        adapts \cite[Corollary 6.4.8]{Muro-Raventos},
        which proves the analogue for triangulated categories.

    \item There is a theorem of Isbell recalled in \cite[Theorem A.5]{Adamek-Rosicky}, stating that $\Set^{\op}$ is a reflective localization of a presentable category if and only if there is no proper class of measurable cardinals. In this case, $\Set^\op$ is very close to presentable—this suggests that bipresentability and cocompactness questions are intimately tied to the existence of large measurable cardinals.

    \item Theorem \ref{introthm:measurable} on the relationship between measurable cardinals and cosumpact objects
        is reflected by examples in abelian groups due to Łoś--Eda
        (cf.~\cite[Theorem III.3.2]{Eklof-Mekler}),
        and in general topology by means of \cite{Keesling-Rudyak},
        see Remarks \ref{rem:slender} and \ref{rem:chaus} for a more detailed discussion.
\end{itemize}

\subsection*{Conventions \& Terminology}

\begin{itemize}
    \item We work with presentable $\infty$-categories as developed in \cite{HTT}.
        What we call a $\kappa$-compact object corresponds to a $\kappa$-presentable object
        in the terminology of \cite{Adamek-Rosicky},
        while what we call a $\kappa$-presentable $\infty$-category corresponds
        to a \emph{locally} $\kappa$-presentable category in the terminology of \cite{Adamek-Rosicky},
        and is called a $\kappa$-compactly generated $\infty$-category in \cite{HTT}.

    \item $\An$ denotes the $\infty$-category of spaces / anima / homotopy types / $\infty$-groupoids.

    \item We call an $\infty$-category $\cC$ copresentable
        if $\cC^\op$ is presentable,
        and bipresentable if it is both presentable and copresentable.

    \item For an $\infty$-category $\cC$ and a regular cardinal $\kappa$,
        we say that an object $X \in \cC$ is $\kappa$-cocompact
        if $X$ is $\kappa$-compact when viewed as an object in $\cC^\op$.
        We denote by $\cC^\kappa$ resp.~$\cC^{\co\kappa}$
        the full subcategories of $\cC$ on the $\kappa$-compact
        resp.~$\kappa$-cocompact objects.

    \item For a cardinal $\kappa$, we write $\kappa^+$ for the smallest
        cardinal larger than $\kappa$, which is always regular.
\end{itemize}

This project is an attempt to grapple with the interactions
between set theory and homotopy theory, as highlighted in the third named author's Barcelona address\footnote{See \url{https://youtu.be/ug3imHlHa4I?t=424}.}.

\subsection*{Acknowledgements}
    We thank Guido Arnone for moral encouragement and Jiří Adámek for
    an email correspondence.
    All authors were funded by the Deutsche Forschungsgemeinschaft (DFG, German Research Foundation) – Project-ID 427320536 – SFB 1442, as well as under Germany’s Excellence Strategy EXC2044/2–390685587, Mathematics Münster: Dynamics–Geometry–Structure.

\section{Presentability and posets}\label{sec:posets}

In this section, we recall the relationship
between (bi)presentability and posets,
which explains the deduction of Corollary \ref{introcor:bipresentable}
from Theorem \ref{introthm:main}.

\begin{proposition}
A poset is a presentable ($\infty$-)category if and only if it is a small complete lattice,
in which case it is a bipresentable ($\infty$-)category.
\end{proposition}
\begin{proof}
    In a poset, the colimit resp.~limit of a diagram is equivalently the join resp.~meet of its objects.
    Thus, a presentable poset is a complete lattice. Furthermore,
    it is small since there is a set (the $\kappa$-compact objects for some regular cardinal
    $\kappa$) of objects such that every object is a small ($\kappa$-filtered) colimit of a
    diagram valued in this set of objects, but this means that every object is a join over
    subsets of a small set, which still only gives a set of objects. Moreover,
    every small complete lattice $L$ admits all (co)limits and every object is both
    $\kappa$-compact and $\kappa$-cocompact for $\kappa > |L|$ since then every $\kappa$-(co)filtered
    diagram is eventually constant, hence it is bipresentable.
\end{proof}

\begin{proposition}\label{prop:conjecture}
    Let $\cC$ be a presentable ($\infty$-)category that satisfies the conclusion of Theorem \ref{introthm:main}.
If $\cC$ is copresentable, then it is a poset.
\end{proposition}
\begin{proof}
    For $\cC$ to be copresentable means that there is a $\kappa\gg0$ such that
    $\cC$ is generated under small limits by a set of $\kappa$-cocompact objects. These are $(-1)$-truncated by the assumption and $(-1)$-truncated objects are stable
    under small limits, so that $\cC$ consists of $(-1)$-truncated objects,
    i.e.~is a poset.
\end{proof}

Thus, if $\cC$ is bipresentable and satisfies the conclusion of Theorem \ref{introthm:main}, it is a small complete lattice by the two propositions above.

\section{The case of ordinary categories}\label{sec:1-cat-case}

In this section, we prove Theorem \ref{introthm:main} for ordinary categories.
The following ad hoc construction is central to the proof;
see Section \ref{sec:tails} for an abstraction and generalization of its properties.

\begin{definition}\label{def:1-tail}
    Let $\cC$ be a bicomplete category, $\kappa$ a regular cardinal, and $X \in \cC$.
    The \textit{$\kappa$-tail} of $X$ is the object
    $L_{\kappa}(X) \coloneqq \lim_{\alpha \in \kappa^{\op}} \coprod_{\kappa \setminus \alpha} X$,
    where the limit is taken along the canonical inclusions,
    equipped with the composite map
    \begin{equation*}
        \begin{tikzcd}
            L_{\kappa}(X) \arrow[r,"\pi_{\emptyset}"] &  \coprod_{\kappa} X \arrow[r,"\nabla"] & X.
        \end{tikzcd}
    \end{equation*}
\end{definition}

\begin{theorem}\label{thm:1cat}
    Let $\cC$ be a bicomplete category such that the pullback functor
    $\mathsf{pb} \colon \Fun(\biglrcorner, \cC) \rightarrow \cC$
    is accessible (e.g.~if $\cC$ is presentable \cite[Proposition 1.59]{Adamek-Rosicky}).
    Then, for every regular cardinal $\kappa$,
    every $\kappa$-cocompact object in $\cC$ is subterminal.
\end{theorem}
\begin{proof}
    Let $Y \in \cC$ be $\kappa$-cocompact for some regular cardinal $\kappa$.
    Since $Y$ is then $\lambda$-cocompact for all regular cardinals $\lambda \ge \kappa$,
    we can assume without loss of generality that pullbacks commute
    with $\kappa$-filtered colimits in $\cC$.
    Let $X \in \cC$ and consider two morphisms $f,g \colon X \rightarrow Y$.
    We need to show that $f = g$.

    Let $L_\kappa(X)$ be the $\kappa$-tail of $X$ in $\cC$
    as in Definition \ref{def:1-tail}.
    For any ordinal $\alpha < \kappa$, we let $P_{\alpha} \coloneqq L_\kappa(X) \times_{\coprod_{\kappa} X} \coprod_{\alpha} X$,
    so that we obtain a $\kappa$-filtered diagram $\kappa \rightarrow \cC,\,\alpha \mapsto P_{\alpha}$ with colimit $L_\kappa(X)$,
    by the assumption that $\kappa$-filtered colimits commute with pullbacks.

    Let $\overline{f},\overline{g} \colon L_\kappa(X) \rightarrow Y$ be defined as
    the composite of the projection $\pi_{\emptyset} \colon L_\kappa(X) \rightarrow \coprod_{\kappa} X$ and
    the maps $\nabla f, \nabla g \colon \coprod_{\kappa} X \rightarrow Y$
    given by $f$ respectively $g$ on every summand. First, we claim that
    $\overline{f} = \overline{g}$. Since $L_\kappa(X) \cong \colim_{\alpha \in \kappa} P_{\alpha}$,
    it suffices to check that $\overline{f}\vert_{P_{\alpha}} = \overline{g}\vert_{P_{\alpha}}$
    for all $\alpha < \kappa$. To this end, let $h_{\alpha} \colon \coprod_{\kappa} X \rightarrow Y$
    be the auxiliary morphism given by $f$ on the first $\alpha$ summands
    and $g$ on the remaining $\kappa \setminus \alpha$ summands.
    Then the two left respectively right composites
    \begin{equation*}
        \begin{tikzcd}
            \coprod_{\alpha} X \arrow[r] & \coprod_{\kappa} X \arrow[r,shift left=0.2em,"\nabla f"]\arrow[r,shift right=0.2em,"h_{\alpha}"'] & Y,
            & \coprod_{\kappa \setminus \alpha} X \arrow[r] & \coprod_{\kappa} X \arrow[r,shift left=0.2em,"h_{\alpha}"]\arrow[r,shift right=0.2em,"\nabla g"'] & Y
        \end{tikzcd}
    \end{equation*}
    are equal by definition, and the diagram
    \begin{equation*}
        \begin{tikzcd}
            P_{\alpha} \arrow[r]\arrow[d] & L_\kappa(X) \arrow[r]\arrow[d] & \coprod_{\kappa \setminus \alpha} X \arrow[dl] \\
            \coprod_{\alpha} X \arrow[r] & \coprod_{\kappa} X &
        \end{tikzcd}
    \end{equation*}
    commutes by construction, witnessing $\overline{f}\vert_{P_{\alpha}} = h_{\alpha}\pi_{\emptyset}\vert_{P_{\alpha}} = \overline{g}\vert_{P_{\alpha}}$.
    So $\overline{f} = \overline{g}$,
    and since the canonical map
    \begin{equation*}
        \begin{tikzcd}
            \colim_{\alpha \in \kappa}\Hom_{\cC}(\coprod_{\kappa \setminus \alpha} X, Y) \arrow[r,"\sim"] & \Hom_{\cC}(L_\kappa(X), Y)
        \end{tikzcd}
    \end{equation*}
    is a bijection by $\kappa$-cocompactness of $Y$,
    there exists some $\alpha < \kappa$ so that
    $\nabla f = \nabla g \colon \coprod_{\kappa \setminus \alpha} X \to Y$.
    But then clearly $f=g$, as desired.
\end{proof}

\begin{corollary}
    Let $\cC$ be a category such that both $\cC$ and $\cC^{\op}$ are presentable.
    Then $\cC$ is a small complete lattice.\qed
\end{corollary}

The proof of Theorem \ref{thm:1cat} serves as motivation
for the proof of Theorem \ref{introthm:main},
which will occupy the next three sections.
The above strategy does not directly generalize to $\infty$-categories,
due to coherence issues.
Specifically, instead of checking that two morphisms are equal,
we need to construct a witnessing homotopy;
while the above proof does construct such homotopies
$\overline{f}|_{P_\alpha} \simeq \overline{g}|_{P_\alpha}$ for each $\alpha < \kappa$,
these cannot be made compatible in a coherent manner.

In fact, the extension of Theorem \ref{thm:1cat} to $\infty$-categories
is false. Indeed, opposites of stable presentable $\infty$-categories
provide a wealth of examples of categories having $\aleph_0$-accessible
pullbacks and a large supply of $\kappa$-cocompact objects.
For instance, $\Sp^\op$ is such a category (cf.~\cite[Section 1.4]{HA}).

\section{Reduction to the pointed case}\label{sec:red-to-pointed}

In this section, we begin the proof of Theorem \ref{introthm:main}
by reducing it to the case where $\cC$ is pointed.

\begin{lemma}\label{lem:pointed-reduction}
    Let $\cC$ be a presentable $\infty$-category, $\kappa$ a regular cardinal,
    and $X \in \cC$.
    We have adjunctions
    \[\begin{tikzcd}[ampersand replacement=\&]
        {\cC_{X//X}} \&\& {\cC_{X/}} \&\& \cC.
        \arrow[""{name=0, anchor=center, inner sep=0}, "U", shift left=2, from=1-1, to=1-3]
        \arrow[""{name=1, anchor=center, inner sep=0}, "{(\pr \colon - \times X \to X)}", shift left=2, from=1-3, to=1-1]
        \arrow[""{name=2, anchor=center, inner sep=0}, "V"', shift right=2, from=1-3, to=1-5]
        \arrow[""{name=3, anchor=center, inner sep=0}, "{(\inc \colon X \to X \amalg -)}"', shift right=2, from=1-5, to=1-3]
        \arrow["\dashv"{anchor=center, rotate=-90}, draw=none, from=0, to=1]
        \arrow["\dashv"{anchor=center, rotate=-90}, draw=none, from=3, to=2]
    \end{tikzcd}\]
    The functor $V$ reflects $\kappa$-cocompact objects,
    and the right adjoint of $U$ preserves them.
\end{lemma}
\begin{proof}
    Since $U$ preserves weakly contractible
    and hence $\kappa$-cofiltered limits,
    its right adjoint preserves $\kappa$-cocompact objects
    by the dual of \cite[Proposition 5.5.7.2]{HTT}.

    To see that $V$ reflects $\kappa$-cocompact objects,
    let $Y \in \cC$ be $\kappa$-cocompact,
    and $f \colon X \to Y$ an object in $\cC_{X/}$.
    Since $\map_{\cC_{X/}}(-,f) \simeq \map_\cC(V-,Y) \times_{\map_\cC(X,Y)} \{f\}$,
    the facts that $V$ creates limits
    and $\kappa$-filtered colimits commute with pullbacks in $\An$
    imply that $f$ is $\kappa$-cocompact.
\end{proof}

\begin{proposition}\label{prop:reduce-to-pointed}
    The following statements are equivalent:
    \begin{enumerate}
        \item Theorem \ref{introthm:main} holds.
        \item Let $\cC$ be a pointed presentable $\infty$-category and $\kappa$
            a regular cardinal. Then every $\kappa$-cocompact object in $\cC$ is zero.
    \end{enumerate}
\end{proposition}
\begin{proof}
    The implication (1) $\Rightarrow$ (2) follows
    from the fact that $(-1)$-truncated objects in pointed $\infty$-categories
    are zero.
    Now suppose that (2) holds. Let $\cC$ be a presentable $\infty$-category,
    $\kappa$ a regular cardinal, $Y \in \cC$ a $\kappa$-cocompact object,
    and $f \colon X \to Y$ any map.
    By Lemma \ref{lem:pointed-reduction},
    the object $A_f \coloneqq (X \xto{(f,\id)} Y \times X \xto{\pr} X)$
    is $\kappa$-cocompact in the pointed presentable $\infty$-category
    $\cC_{X//X}$.
    But then, by assumption, $A_f$ is zero in $\cC_{X//X}$,
    and we deduce
    \[
        *
        \simeq \map_{\cC_{X//X}}(X \xto{\inc} X \amalg X \xto{\nabla} X, A_f)
        \simeq \map_{\cC_{X/}}(X \xto{\inc} X \amalg X, f)
        \simeq \map_{\cC}(X,Y)
    \]
    using both adjunctions of Lemma \ref{lem:pointed-reduction}.
    So $\map_\cC(X,Y)$ is contractible if it is non-empty, as desired.
\end{proof}

\section{Phantom maps in pointed \texorpdfstring{$\infty$}{∞}-categories}\label{sec:phantom}

In this section, we prove a nilpotence result for phantom maps\footnote{Phantom maps originally arose in algebraic topology \cite{Adams-Walker, Gray}
	and have since been studied in the setting of triangulated categories \cite{Neeman}.
	Closely related is the theory of projective classes introduced in \cite{Christensen}. }
in pointed presentable $\infty$-categories (Proposition \ref{prop:phantom-nilpotent}).
Concretely, we show that the ideal of maps which can be written
as $n$-fold composites of $\kappa$-phantom maps for every $n$
squares to zero when $\cC$ is $\kappa$-presentable.
This generalizes a result of Muro--Raventós \cite[Corollary 6.4.8]{Muro-Raventos},
who prove the result for triangulated categories,
and our proof is an adaptation of theirs. The relevance to Theorem \ref{introthm:main} is that
we construct a sufficient supply of $\infty$-fold $\kappa$-phantom maps
to test against a $\kappa$-cocompact object in Section \ref{sec:tails} below.

Throughout this section, $\kappa$ is a regular cardinal
and $\cC$ a pointed cocomplete $\infty$-category.

\begin{definition}
    A map $f \colon X \rightarrow Y$ in $\cC$ is called
    \textit{$\kappa$-phantom} if for every map $c \colon C \rightarrow X$
    with $C$ a $\kappa$-compact object,
    the composite $fc$ is the zero map.
    A map that can be written as a composite of $n$ $\kappa$-phantom maps
    is called \textit{$n$-fold $\kappa$-phantom}
    and the collection of these will be denoted $\Ph_{\kappa}^n$.
    We refer to maps in the intersection $\Ph_{\kappa}^{\infty} \coloneqq \bigcap_{n\ge1} \Ph_{\kappa}^n$
    as \textit{$\infty$-fold $\kappa$-phantom} maps.
\end{definition}

\begin{observation}
    The collections $\Ph_{\kappa}^n$ for $1 \leq n \leq \infty$ are \textit{ideals}
    in the sense that for $(f \colon X \to Y) \in \Ph^n_\kappa$
    and arbitrary maps $e \colon W \rightarrow X, g \colon Y \rightarrow Z$ in $\cC$, we also have $gfe \in \Ph_{\kappa}^n$.
\end{observation}

\begin{observation}
    If $\lambda \geq \kappa$ is another regular cardinal,
    then any $\lambda$-phantom map is also $\kappa$-phantom.
\end{observation}

\begin{observation}
    Let $L \colon \cC \rightleftarrows \cD \colon R$
    be an adjunction of pointed $\infty$-categories.
    Then, if the left adjoint $L$ preserves $\kappa$-compact objects, the right adjoint $R$ preserves $\kappa$-phantom maps.
\end{observation}

\begin{definition}
    Given classes of maps $\cP, \cQ$ in $\cC$, we let $\cP\circ \cQ \coloneqq \{p\circ q \mid p \in \cP,\ q \in \cQ\}$ and
    \[
        \lperp\cP \coloneqq
        \{X \in \cC \mid \text{for $(\phi \colon Y \to Z) \in \cP$
        and any map $f \colon X \to Y$, we have $\phi f \simeq 0$}\}.
    \]
\end{definition}

\begin{lemma}\label{lem:extension}
    Let $\cP, \cQ$ be two classes of maps in $\cC$. If $A \to B \to C$ is a cofiber sequence
    with $A \in \lperp\cP$ and $C \in \lperp\cQ$, then $B \in \lperp{(\cQ \circ \cP)}$.
\end{lemma}
\begin{proof}
    Let $A \to B \to C$ be as in the statement,
    and consider maps $B \xto{f} X \xto{\phi} Y \xto{\psi} Z$
    where $\phi \in \cP$ and $\psi \in \cQ$.
    We have the following commutative diagram
    \[\begin{tikzcd}
        A & B & X & Y \\
        0 & C && Z,
        \arrow[from=1-1, to=1-2]
        \arrow["0"{description}, shift left, curve={height=-18pt}, from=1-1, to=1-4]
        \arrow[from=1-1, to=2-1]
        \arrow["f", from=1-2, to=1-3]
        \arrow[from=1-2, to=2-2]
        \arrow["\phi", from=1-3, to=1-4]
        \arrow["\psi", from=1-4, to=2-4]
        \arrow[from=2-1, to=2-2]
        \arrow["\lrcorner"{anchor=center, pos=0.125, rotate=180}, draw=none, from=2-2, to=1-1]
        \arrow["g"{description}, dashed, from=2-2, to=1-4]
        \arrow["0"', from=2-2, to=2-4]
    \end{tikzcd}\]
    where we note that the restriction of $\phi f$ to $A$ vanishes since $A \in \lperp\cP$, so $\phi f$ factors through some map $g \colon C \to Y$.
    But then also $\psi g$ vanishes since $C \in \lperp\cQ$.
    Hence $\psi \phi f \simeq 0$, as desired.
\end{proof}

\begin{definition}
    Let $\Cell^0_\kappa$ be the full subcategory of $\cC$
    on arbitrary coproducts of $\kappa$-compact objects
    and inductively let
    \[
        \Cell^{n+1}_\kappa \coloneqq \{B \amalg_A C \mid A,B,C \in \Cell^n_\kappa\}
    \]
    be the full subcategory on objects which can be written
    as pushouts of objects in $\Cell^n_\kappa$.
\end{definition}

\begin{lemma}\label{lem:cellular}~
    \begin{enumerate}
        \item For $1 \leq n \leq \infty$,
            the subcategory $\APh^n_\kappa \subseteq \cC$
            is closed under coproducts.

        \item $\Cell^n_\kappa$ contains $\cC^\kappa$
            and is closed under coproducts and suspensions
            for $n \geq 0$.

        \item $\Cell^n_\kappa \subseteq \APh^{2^n}_\kappa$
            for $n \geq 0$.

        \item Let $\langle \Cell^0_\kappa \rangle \subseteq \cC$
            be the full subcategory generated by $\Cell^0_\kappa$ under finite
            colimits.
            Then
            \[
                \langle \Cell^0_\kappa\rangle
                = \bigcup_{n \geq 0}\Cell^n_\kappa
                \subseteq \APh^\infty_\kappa.
            \]
    \end{enumerate}
\end{lemma}
\begin{proof}~
\begin{enumerate}
    \item This is clear.

    \item This follows from a straightforward induction, the base case $n=0$ being true by definition.

    \item We do this by induction, the case $n=0$ being clear
        by definition of $\Cell^0_\kappa$ and (1).
        Suppose that $\Cell^n_\kappa \subseteq \APh^{2^n}_\kappa$.
        Let $B \leftarrow A \rightarrow C$ be a span in $\Cell^n_\kappa$ and $P \coloneqq B \amalg_A C$.
        We can rewrite this as $P \simeq (B \vee C)\amalg_{A \vee A} A$,
        which yields a cofiber sequence $B \vee C \to P \to \Sigma A$ by pushout pasting
        and the fact that $\cofib(\nabla \colon A \vee A \to A) \simeq \Sigma A$.
        But now $B \vee C \in \APh^{2^n}_\kappa$ by (1),
        and $\Sigma A \in \Cell^{n}_\kappa \subseteq \APh^{2^n}_\kappa$ by (2).
        Thus Lemma \ref{lem:extension} yields $P \in \APh^{2^{n+1}}_\kappa$,
        concluding the induction.

    \item For the inclusion,
        we clearly have $\APh^n_\kappa \subseteq \APh^\infty_\kappa$
        for all $n \geq 1$, hence this follows from (3).
        For the equality, the inclusion $\supseteq$ is clear,
        so it remains to show that $\bigcup_{n \geq 0}\Cell^n_\kappa$
        is already closed under finite colimits.
        Since each $\Cell^{n}_\kappa$ is closed under coproducts,
        we see that the union admits finite coproducts.
        Since it also admits pushouts by construction
        (any three objects in the union lie in some common stage),
        we conclude.\qedhere
\end{enumerate}
\end{proof}

\begin{lemma}\label{lem:pointed-seq-colim}
    Let $X_\bullet \colon \N \to \cC$ be a sequential diagram.
    Then there is a cofiber sequence
    \[
        \bigvee_n X_n \to \colim_n X_n \to \bigvee_n \Sigma X_n.
    \]
\end{lemma}
\begin{proof}
    It is standard that we can write $\colim_n X_n$
    as the coequalizer of $\id,\text{shift} \colon \bigvee_n X_n \rightrightarrows \bigvee_n X_n$.
    Rewriting this as a pushout, we obtain the following commutative diagram:
    \[\begin{tikzcd}
        {\bigvee_n X_n \vee \bigvee_nX_n} & {\bigvee_n X_n} & 0 \\
        {\bigvee_n X_n} & {\colim_n X_n} & {\cofib(\nabla).}
        \arrow["{\id \vee \text{shift}}", from=1-1, to=1-2]
        \arrow["\nabla"', from=1-1, to=2-1]
        \arrow[from=1-2, to=1-3]
        \arrow[from=1-2, to=2-2]
        \arrow[from=1-3, to=2-3]
        \arrow[from=2-1, to=2-2]
        \arrow["\lrcorner"{anchor=center, pos=0.125, rotate=180}, draw=none, from=2-2, to=1-1]
        \arrow[from=2-2, to=2-3]
        \arrow["\lrcorner"{anchor=center, pos=0.125, rotate=180}, draw=none, from=2-3, to=1-2]
    \end{tikzcd}\]
    Finally, $\cofib(\nabla) \simeq \bigvee_n \cofib(\nabla \colon X_n \vee X_n \to X_n) \simeq \bigvee_n \Sigma X_n$
    by noting that $\nabla$ agrees with the componentwise codiagonal and commuting coproducts with pushouts.
\end{proof}

\begin{proposition}\label{prop:phantom-nilpotent}
    Let $\cC$ be a pointed $\kappa$-presentable $\infty$-category.
    Then $(\Ph_\kappa^\infty)^2 = 0$,
    i.e.~if $\phi,\psi \in \Ph_\kappa^\infty$ are composable,
    then $\psi\phi \simeq 0$.
\end{proposition}
\begin{proof}
    Let $\phi \colon X \to Y$ and $\psi \colon Y \to Z$ be the two
    maps in $\Ph^\infty_\kappa$.
    Since $\cC$ is $\kappa$-presentable,
    we can write $X \simeq \colim(\cC^\kappa_{/X} \to \cC^\kappa \to \cC)$,
    and applying the Bousfield--Kan formula (see e.g.~\cite[Corollary 12.3]{Shah}),
    we obtain a simplicial diagram $X_\bullet \colon \Delta^\op \to \cC$
    with $X \simeq |X_\bullet|$,
    where each $X_n$ is a coproduct of $\kappa$-compact objects,
    hence lies in $\Cell^0_\kappa$.
    Using the skeletal filtration of $X_\bullet$,
    we have
    \[
        |\sk_n X_\bullet| \simeq \colim(\Delta^\op_{\leq n} \subseteq \Delta^\op \xto{X_\bullet} \cC)
        \quad\text{and}\quad
        X \simeq |X_\bullet| \simeq \colim_n |\sk_n X_\bullet|.
    \]
    In particular, each $|\sk_n X_\bullet|$ lies in the finite colimit closure
    of $\{X_n \mid n \geq 0\} \subseteq \Cell^0_\kappa$ (see e.g.~\cite[Example 6.5.3]{nine}),
    so we deduce that $|\sk_n X_\bullet| \in \langle \Cell^0_\kappa\rangle$
    in the notation of Lemma \ref{lem:cellular}.
    Then also $\Sigma |\sk_n X_\bullet| \in \langle \Cell^0_\kappa\rangle$,
    and hence by parts (1) and (4) of that lemma both
    $\bigvee_n |\sk_n X_\bullet|$ and $\bigvee_n \Sigma |\sk_n X_\bullet|$
    lie in $\APh^\infty_\kappa$.
    From Lemmas \ref{lem:extension} and \ref{lem:pointed-seq-colim},
    we deduce that $X \in \lperp{((\Ph^\infty_\kappa)^2)}$, hence $\psi\phi \simeq 0$.
\end{proof}

\section{Tail functors and Theorem \ref{introthm:main}}\label{sec:tails}

In this section, we upgrade the key observations in the proof of
Theorem \ref{thm:1cat} by defining the general notion of a
\textit{tail functor}, generalizing the construction of
Definition \ref{def:1-tail}.
Combining this with the nilpotence result (Proposition \ref{prop:phantom-nilpotent}) of the previous section, we will prove Theorem \ref{introthm:main}.

\begin{definition}
        Let $\cC$ be a pointed bicomplete $\infty$-category, $\kappa$ a regular cardinal, and $1 \leq n \leq \infty$.
        An \textit{$n$-fold $\kappa$-tail functor} on $\cC$ is the datum of a natural transformation of the form
        \[\begin{tikzcd}[column sep=small]
            {\cC } && {\Fun(\cI, \cC)} && \cC
            \arrow[""{name=0, anchor=center, inner sep=0}, "T", curve={height=-18pt}, from=1-1, to=1-3]
            \arrow[""{name=1, anchor=center, inner sep=0}, "\const"', curve={height=18pt}, from=1-1, to=1-3]
            \arrow["{\lim_\cI}", from=1-3, to=1-5]
            \arrow[shorten <=5pt, shorten >=5pt, Rightarrow, from=0, to=1]
        \end{tikzcd}\]
    subject to the following conditions:
    \begin{itemize}
        \item The category $\cI$ is small and $\kappa$-cofiltered.
        \item For every $X \in \cC$, the induced map $\lim_{\cI} TX \rightarrow \lim_{\cI} X \simeq X$ is an $n$-fold $\kappa$-phantom map.
        \item For every $X \in \cC$, the transformation $TX \rightarrow \const X$ admits pointwise sections.
    \end{itemize}
\end{definition}

This definition serves to provide functorial constructions of $n$-fold $\kappa$-phantom maps
via $\kappa$-cofiltered limits.
For any $X \in \cC$, the transformation
$TX \rightarrow \const X$ of functors $\cI \rightarrow \cC$ equivalently
encodes a diagram $\cI \rightarrow \cC_{/X}$ lifting $TX$.
The third of the above conditions says that this diagram takes values in the full subcategory
of $\cC_{/X}$ on those objects $Y \rightarrow X$ admitting sections.
This condition is crucial due to the next proposition.

\begin{proposition}\label{prop:pi0-inj}
    Let $T \colon \cI \rightarrow \cC_{/X}$ be a $\kappa$-cofiltered diagram such that
    the structure map $Ti \rightarrow X$ admits a section for all $i \in \cI$.
    Then, for any $\kappa$-cocompact object $Y$ in $\cC$, the pre-composition map
    \begin{equation*}
            \map_{\cC}(X, Y) \rightarrow \map_{\cC}(\lim_{\cI} Ti, Y) \simeq \colim_{\cI^{\op}} \map_{\cC}(Ti, Y)
    \end{equation*}
    is $\pi_0$-injective.
\end{proposition}
\begin{proof}
    This map is a $\kappa$-filtered colimit of maps that admit retractions,
    hence its image under $\pi_0$ is a $\kappa$-filtered colimit of injections
    in $\Set$ and thus an injection.
\end{proof}

We now turn to proving the existence of $n$-fold $\kappa$-tail functors
for all regular cardinals $\kappa$ and $1 \leq n \leq \infty$.
For $n=1$, this is implemented by the explicit construction of Definition \ref{def:1-tail}.
We proceed by induction to prove this for all finite $n$,
and then take a suitable limit to obtain the case $n=\infty$.

\begin{proposition}\label{prop:1-tail}
    For any pointed bicomplete $\infty$-category $\cC$ and regular cardinal $\kappa$,
    there exists a ($1$-fold) $\kappa$-tail functor on $\cC$.
    Explicitly, such a functor is given by the diagram
    \[\begin{tikzcd}[column sep=small]
        {\cC } && {\Fun(\kappa^\op, \cC)} && {\cC,} & {S_\kappa(X) \colon \kappa^\op \to \cC,\ \alpha \mapsto \bigvee_{\kappa \setminus \alpha} X}.
        \arrow[""{name=0, anchor=center, inner sep=0}, "{S_\kappa}", curve={height=-18pt}, from=1-1, to=1-3]
        \arrow[""{name=1, anchor=center, inner sep=0}, "\const"', curve={height=18pt}, from=1-1, to=1-3]
        \arrow["{\lim_{\kappa^\op}}", from=1-3, to=1-5]
        \arrow["\nabla"', shorten <=5pt, shorten >=5pt, Rightarrow, from=0, to=1]
    \end{tikzcd}\]
\end{proposition}

\begin{proof}
    Indeed, $\kappa^{\op}$ is $\kappa$-cofiltered and the fold maps
    $\nabla \colon \bigvee_{\kappa \setminus \alpha} X \rightarrow X$
    admit sections since $\kappa \setminus \alpha \neq \emptyset$
    for all $\alpha \in \kappa$.
    Finally, the induced map on limits may be written as the composite
    \begin{equation*}
        \begin{tikzcd}
            L_\kappa(X) \coloneqq \lim_{\alpha \in \kappa^{\op}} \bigvee_{\kappa \setminus \alpha} X \arrow[r,"\pi_{\emptyset}"] & \bigvee_{\kappa} X \arrow[r,"\nabla"] & X.
        \end{tikzcd}
    \end{equation*}
    Here, already the map $\pi_\emptyset$ is $\kappa$-phantom.
    Namely, if $f \colon C \to L_\kappa(X)$ is any map from a $\kappa$-compact object $C$,
    then we construct the commutative diagram
    \[\begin{tikzcd}
        C & {L_\kappa (X)} & {\bigvee_{\kappa \setminus \beta}X} \\
        {\bigvee_\beta X} & {\bigvee_{\kappa} X}
        \arrow["f", from=1-1, to=1-2]
        \arrow["g"', dashed, from=1-1, to=2-1]
        \arrow["{\pi_\beta}", from=1-2, to=1-3]
        \arrow["{\pi_\emptyset}"', from=1-2, to=2-2]
        \arrow["{j_\beta}", from=1-3, to=2-2]
        \arrow["{i_\beta}"', from=2-1, to=2-2]
    \end{tikzcd}\]
    by using $\kappa$-compactness to factor $\pi_{\emptyset} f$ through $i_\beta$ for some $\beta < \kappa$,
    and then noting that $\pi_\emptyset$ factors through $\pi_\beta$.
    Now $i_\beta$ admits a retraction $0 \vee \id \colon \bigvee_\kappa X \simeq \bigvee_{\kappa \setminus \beta} X \vee \bigvee_{\beta} X \to \bigvee_{\beta}X$,
    which vanishes upon precomposing with $j_\beta$,
    so a diagram chase lets us deduce that $\pi_\emptyset f \simeq 0$.
\end{proof}

\begin{remark}\label{rem:semiadditive-tail}
    If $\cC$ is semiadditive,
    then $\lim_{\aleph_0^\op} S_{\aleph_0}(X) \simeq \fib(\bigoplus_{\aleph_0} X \to \prod_{\aleph_0} X)$.
\end{remark}

\begin{proposition}
    Let $\cC$ be a pointed bicomplete $\infty$-category
    and $\kappa$ a regular cardinal.
    If there exists an $n$-fold $\kappa$-tail functor on $\cC$,
    then there exists an $(n+1)$-fold $\kappa$-tail functor on $\cC$.
\end{proposition}
\begin{proof}
    Let $(T \Rightarrow \const) \colon \cC \to \Fun(\cI, \cC)$
    be an $n$-fold $\kappa$-tail functor on $\cC$.
    Pick a regular cardinal $\lambda \geq \kappa$
    such that $\cI$ is $\lambda$-small and let
    $(S_\lambda \Rightarrow \const) \colon \Fun(\cI,\cC) \to \Fun(\lambda^\op,\Fun(\cI,\cC))$
    be the ($1$-fold) $\lambda$-tail functor on $\Fun(\cI,\cC)$
    provided by Proposition \ref{prop:1-tail}.
    Now consider the diagram
    \[\begin{tikzcd}[column sep=scriptsize]
        {\cC } && {\Fun(\cI, \cC)} && {\Fun(\lambda^\op, \Fun(\cI,\cC))} && {\Fun(\cI,\cC)} && \cC.
        \arrow[""{name=0, anchor=center, inner sep=0}, "T", curve={height=-18pt}, from=1-1, to=1-3]
        \arrow[""{name=1, anchor=center, inner sep=0}, "\const"', curve={height=18pt}, from=1-1, to=1-3]
        \arrow[""{name=2, anchor=center, inner sep=0}, "{S_\lambda}", curve={height=-18pt}, from=1-3, to=1-5]
        \arrow[""{name=3, anchor=center, inner sep=0}, "\const"', curve={height=18pt}, from=1-3, to=1-5]
        \arrow["{\lim_{\lambda^\op}}", from=1-5, to=1-7]
        \arrow["{\lim_\cI}", from=1-7, to=1-9]
        \arrow[shorten <=5pt, shorten >=5pt, Rightarrow, from=0, to=1]
        \arrow[shorten <=5pt, shorten >=5pt, Rightarrow, from=2, to=3]
    \end{tikzcd}\]
    We claim that this witnesses the composite
    $(S_\lambda T \Rightarrow \const_{\lambda^{\op}\times\cI}) \colon \cC \rightarrow \Fun(\lambda^{\op} \times \cI, \cC)$
    as an $(n+1)$-fold $\kappa$-tail functor on $\cC$.

    Indeed, $\lambda^\op \times \cI$ is still $\kappa$-cofiltered.
    Furthermore, $TX \to \const_{\cI}X$ and hence
    $\const_{\lambda^\op}TX \to \const_{\lambda^\op\times \cI}X$
    admit pointwise sections
    by the assumption on $T$,
    and $S_\lambda TX \to \const_{\lambda^\op}TX$
    admits pointwise sections by the assumption on $S_\lambda$.
    Therefore also the composite
    $S_\lambda TX \to \const_{\lambda^\op \times \cI}X$
    admits pointwise sections.

    Finally, we have to show that for $X \in \cC$ the map $\lim_{\lambda^\op \times \cI} S_\lambda TX \to X$ is an $(n+1)$-fold $\kappa$-phantom map.
    The map  $\lim_{\lambda^\op} S_\lambda TX \to TX$
    is $\lambda$-phantom in $\Fun(\cI,\cC)$
    by the assumption on $S_\lambda$.
    Moreover, by \cite[Proposition 5.3.4.13]{HTT},
    $\const \colon \cC \to \Fun(\cI,\cC)$
    preserves $\lambda$-compact objects,
    so that the functor $\lim_\cI$
    preserves $\lambda$-phantom maps.
    It follows that
    \[
        \lim_{\lambda^{\op} \times \cI} S_{\lambda} TX \simeq
        \lim_\cI \lim_{\lambda^\op} S_\lambda TX
        \to \lim_\cI TX
        \to X
    \]
    is the composite of a $\lambda$-phantom map
    followed by an $n$-fold $\kappa$-phantom map by the assumption on $T$.
    Since $\lambda \geq \kappa$, this composite
    is an $(n+1)$-fold $\kappa$-phantom map.
\end{proof}

\begin{proposition}
    Let $\cC$ be a pointed bicomplete $\infty$-category and $\kappa$ a regular cardinal.
    For each $n\ge1$, let $(T_n \Rightarrow \const_{\cI_n}) \colon \cC \rightarrow \Fun(\cI_n, \cC)$ be an $n$-fold $\kappa$-tail functor on $\cC$.
    Then the pullback square
    \begin{equation*}
        \begin{tikzcd}
            T_{\infty}\arrow[d, Rightarrow]\arrow[r, Rightarrow]\arrow[dr,phantom,very near start,"\lrcorner"] & \prod_{n\ge1} T_n\arrow[d, Rightarrow]\\
            \const_{\prod_{n\ge1}\cI_n}\arrow[r, Rightarrow, "\Delta"] & \prod_{n\ge1} \const_{\cI_n}
        \end{tikzcd}
    \end{equation*}
    defines an $\infty$-fold $\kappa$-tail functor $(T_{\infty} \Rightarrow \const) \colon \cC \rightarrow \Fun(\prod_{n\ge1} \cI_n, \cC)$.
\end{proposition}
\begin{proof}
    The category $\prod_{n\ge1} \cI_n$ is $\kappa$-cofiltered as a product of $\kappa$-cofiltered categories.
    For any $X \in \cC$, fix $\ell \geq 1$ and consider the following diagram of functors $\prod_{n\ge1} \cI_n \rightarrow \cC$:
    \begin{equation*}
        \begin{tikzcd}
            T_{\infty}X\arrow[r]\arrow[d]\arrow[dr,phantom,very near start,"\lrcorner"] & \prod_{n\ge1} T_nX\arrow[r,"\pi_\ell"]\arrow[d] & T_\ell X\arrow[d]\\
            \const_{\prod_{n \geq 1}\cI_n}X\arrow[r,"\Delta"] & \prod_{n\ge1} \const_{\cI_n} X\arrow[r,"\pi_\ell"] & \const_{\prod_{n \geq 1}\cI_n} X.
        \end{tikzcd}
    \end{equation*}
    The vertical transformation $\prod_{n\ge1}T_nX \rightarrow \prod_{n\ge1}\const_{\cI_n} X$
    admits pointwise sections as it is a product
    of transformations that admit pointwise sections,
    hence its base-change $T_{\infty}X \rightarrow \const_{\prod_{n \geq 1}\cI_n}X$
    also admits pointwise sections.

    Furthermore, the bottom horizontal composite is the identity,
    so taking limits and using that
    $\pi_\ell \colon \prod_{n\ge1} \cI_n \rightarrow \cI_\ell$ is coinitial since
    $\kappa$-cofiltered categories are weakly contractible,
    we see that $\lim_{\prod_{n\ge1} \cI_n} T_{\infty}X \rightarrow X$ factors through
    the $\ell$-fold $\kappa$-phantom map $\lim_{\cI_\ell} T_\ell X \rightarrow X$,
    hence is itself an $\ell$-fold $\kappa$-phantom map.
    As $\ell\ge1$ was arbitrary, we conclude that
    it is an $\infty$-fold $\kappa$-phantom map.
\end{proof}

\begin{corollary}\label{cor:infty-tail}
    Let $\cC$ be a pointed bicomplete $\infty$-category.
    For all regular cardinals $\kappa$,
    there exists an $\infty$-fold $\kappa$-tail functor on $\cC$.\qed
\end{corollary}

\begin{proposition}\label{prop:main-pointed}
    Let $\cC$ be a pointed presentable $\infty$-category and $\kappa$ a regular cardinal.
    Then every $\kappa$-cocompact object in $\cC$ is zero.
\end{proposition}
\begin{proof}
    Let $X \in \cC$ be $\kappa$-cocompact.
    After possibly increasing $\kappa$,
    we may assume that $\cC$ is $\kappa$-presentable.
    Let $(T \Rightarrow \const) \colon \cC \rightarrow \Fun(\cI, \cC)$ be
    an $\infty$-fold $\kappa$-tail functor, which exists by Corollary \ref{cor:infty-tail},
    and let $(L \Rightarrow \id) \colon \cC \rightarrow \cC$ be its limit.
    Proposition \ref{prop:phantom-nilpotent} then shows that the composite
    \[
        LLX \to LX \to X
    \]
    is the zero map.
    On the other hand, applying $\map_{\cC}(-, X)$ to this composite
    yields a $\pi_0$-injective map by using Proposition \ref{prop:pi0-inj} twice.
    Together, this implies $\id_X \simeq 0$, hence $X \simeq 0$.
\end{proof}

\begin{theorem}\label{thm:main}
    Let $\cC$ be a presentable $\infty$-category and $\kappa$ a regular cardinal.
    Then every $\kappa$-cocompact object in $\cC$ is $(-1)$-truncated.
\end{theorem}
\begin{proof}
    Combine Propositions \ref{prop:reduce-to-pointed} and \ref{prop:main-pointed}.
\end{proof}

\begin{corollary}\label{cor:bipresentable}
    Let $\cC$ be an $\infty$-category such that both $\cC$ and $\cC^\op$
    are presentable. Then $\cC$ is a small complete lattice.\qed
\end{corollary}

It turns out that in some cases the limit of the tail functor
from Proposition \ref{prop:1-tail} is already initial,
allowing one to give a simpler proof of a stronger result in this case.

\begin{proposition}
    Let $\cC$ be a bicomplete $\infty$-category
    such that $\mathsf{pb} \colon \Fun(\biglrcorner, \cC) \to \cC$ is accessible
    and $\cC$ is extensive, i.e.~$\cC_{/\coprod_i X_i} \xto{\simeq} \prod_i \cC_{/X_i}$ for all finite collections of objects $(X_i)_i$
    (e.g.~if $\cC$ is an $\infty$-topos, $\Cat_\infty$, or $\mathsf{Cond}(\An)$).
    Then, for every regular cardinal $\kappa$, every $\kappa$-cocompact object in $\cC$ is terminal.
\end{proposition}
\begin{proof}
    Up to increasing $\kappa$,
    we may assume that pullbacks commute with $\kappa$-filtered colimits.
    Then
    \[
        \lim_{\alpha \in \kappa^\op}\coprod_{\kappa \setminus \alpha}X
        \eqqcolon L_\kappa(X)
        \simeq \colim_{\beta < \kappa} \coprod_{\beta} X \times_{\coprod_\kappa X} L_\kappa(X)
        \simeq \colim_{\beta < \kappa} \lim_{\beta < \alpha \in \kappa^\op}
        \coprod_{\beta} X \times_{\coprod_\kappa X}\coprod_{\kappa \setminus \alpha} X
        \simeq \empty,
    \]
    where the last equivalence uses that $\coprod_{\beta}X \times_{\coprod_\kappa X} \coprod_{\kappa \setminus \alpha}X \simeq \empty$
    for $\alpha > \beta$ due to extensivity.
    In particular, if $Y \in \cC$ is $\kappa$-cocompact
    and $X \in \cC$ arbitrary, we obtain
    \[
        * \simeq \map_{\cC}(L_\kappa(X), Y)
        \simeq \colim_{\alpha < \kappa} \prod_{\kappa \setminus \alpha}\map_{\cC}(X,Y).
    \]
    So $\map_{\cC}(X,Y)$ is non-empty.
    Since $\pi_0$ commutes with products and filtered colimits,
    we deduce $* \cong \colim_{\alpha < \kappa} \prod_{\kappa \setminus \alpha} \pi_0\map_{\cC}(X,Y)$,
    which is impossible for $|\pi_0\map_{\cC}(X,Y)|\geq 2$.
    Thus $\pi_0\map_{\cC}(X,Y) \cong *$, and the same argument with $\pi_n$ for $n \geq 1$
    then shows that $\pi_*\map_{\cC}(X,Y) \cong *$, so $\map_{\cC}(X,Y) \simeq *$.
\end{proof}

Finally, we derive a curious structural result about presentable $\infty$-categories.

\begin{proposition}
    Let $\cC$ be a presentable $\infty$-category and
    $L(\cC) \coloneqq \bigcup_{\kappa} \cC^{\co\kappa} \subseteq \cC$
    be the full subcategory on objects which are $\kappa$-cocompact
    for some regular cardinal $\kappa$.
    Then $L(\cC)$ is a small complete lattice
    and its inclusion admits a coaccessible left adjoint
    $p \colon \cC \rightarrow L(\cC)$.
\end{proposition}
\begin{proof}
    The full subcategory $\tau_{\le-1}\cC \subseteq \cC$
    on the $(-1)$-truncated objects is an accessible localization
    by \cite[Proposition 5.5.6.18]{HTT}.
    In particular $\tau_{\le-1}\cC$ is a presentable poset,
    i.e.~a small complete lattice.
    By Theorem \ref{thm:main}, we have $L(\cC) \subseteq \tau_{\le-1}\cC$,
    hence this is also a small poset.
    Moreover, for any small diagram $D \colon \cI \rightarrow L(\cC)$,
    there exists a sufficiently large regular cardinal $\kappa$
    such that $\cI$ is $\kappa$-small and $D$ factors through
    the full subcategory on the $\kappa$-cocompact objects.
    This implies $\lim_{\cI} D$ is again $\kappa$-cocompact,
    hence $L(\cC) \subseteq \tau_{\le-1}\cC$ is closed under small limits,
    which  implies that $L(\cC)$ is a small complete lattice
    and that its inclusion into $\tau_{\le-1}\cC$ admits a left adjoint
    by the dual adjoint functor theorem.
    In total, we obtain a composite left adjoint
    $p \colon \cC \rightarrow L(\cC)$.

    This left adjoint $p$ maps an object $X$ to the minimal object
    $pX$ in $L(\cC)$ receiving a (necessarily unique) map $X \rightarrow pX$.
    Since $L(\cC)$ is small, there exists a sufficiently
    large regular cardinal $\kappa$ so that $L(\cC) = \cC^{\co\kappa}$.
    We claim that $p$ is $\kappa$-coaccessible.
    If $D \colon \cI \rightarrow \cC$ is a $\kappa$-cofiltered diagram,
    then $p\lim_{\cI}D \le \lim_{\cI}pD$.
    However, the map $\lim_{\cI}D \rightarrow p\lim_{\cI}D$
    factors through $\pi_i \colon \lim_{\cI}D \rightarrow Di$
    for some $i \in \cI$ since $\cI$ is $\kappa$-cofiltered and
    the target is a $\kappa$-cocompact object.
    This implies $\lim_{\cI}pD \le pDi \le p\lim_{\cI}D$,
    so all the inequalities are equalities.
\end{proof}

\begin{remark}
    If $P$ is a small complete lattice, then $L(P) \cong P$.
    Furthermore, observe that if $\cC_i,\,i \in I$ is a set of presentable $\infty$-categories,
    then $L(\prod_{i \in I} \cC_i) \simeq \prod_{i \in I} L(\cC_i)$.
    Also, if $\cC,\cD$ are presentable $\infty$-categories,
    then $L(\cC \star \cD) \cong L(\cC) \star L(\cD)$
    (indeed, the join $\cC \star \cD$ is again a presentable $\infty$-category).
\end{remark}

\begin{remark}
    Furthermore, it is possible to show that $\cC \mapsto L(\cC)$ assembles into a left adjoint
    to the inclusion of the category of small complete lattices and left adjoints
    into the category of presentable $\infty$-categories and coaccessible left adjoints.
\end{remark}

\section{Cosumpactness and measurable cardinals}\label{sec:measurable}

In this section, we introduce a weakening of $\kappa$-compactness
we refer to as $\kappa$-sumpactness, which only tests against coproducts
instead of general $\kappa$-filtered colimits.
In Subsection \ref{subsec:distr}, we show that the analogue of Theorem \ref{introthm:main} for $\kappa$-cosumpact objects holds for infinitary distributive $\infty$-categories
(Proposition \ref{prop:distr-case}),
and in Subsection \ref{subsec:measurable}
we prove Theorem \ref{introthm:measurable}.
We now define the notion of $\kappa$-sumpactness.

\begin{definition}\label{def:cosumpact}
    Let $\cC$ be an $\infty$-category with small coproducts and $\kappa$ a regular cardinal.
    We say that $X \in \cC$ is \textit{$\kappa$-sumpact} if for any set $I$,
    any $\kappa$-filtered union $I = \bigcup_{j \in J} I_j$
    of subsets $I_j \subseteq I$,
    and any collection of objects $Y_i \in \cC$,\,$i \in I$,
    the following canonical map is an equivalence:
    \[
        \colim_{j \in J} \map_\cC(X,\coprod_{i \in I_j}Y_i) \to \map_\cC(X, \colim_{j \in J}\coprod_{i \in I_j}Y_i) \simeq \map_\cC(X, \coprod_{i \in I} Y_i).
    \]
    Dually, and more importantly for us, if $\cC$ admits small products,
    then we say $X \in \cC$ is \textit{$\kappa$-cosumpact} if $X$ is $\kappa$-sumpact in $\cC^\op$, which means that the map
    \[
        \colim_{j \in J} \map_\cC(\prod_{i \in I_j}Y_i, X)
        \to \map_\cC(\lim_{j \in J^\op}\prod_{i \in I_j}Y_i, X)
        \simeq \map_\cC(\prod_{i \in I}Y_i, X)
    \]
    is an equivalence.
\end{definition}

We can always write $I$ as the $\kappa$-filtered union
of its $\kappa$-small subsets,
so in particular any map from a $\kappa$-sumpact
object into a coproduct factors through a $\kappa$-small ``sub-coproduct''.

\begin{convention}
    We will abbreviate ``$\aleph_0$-(co)sumpact'' to ``(co)sumpact''.
\end{convention}

\begin{remark}
    Any $\kappa$-compact object is $\kappa$-sumpact,
    and for $\lambda \geq \kappa$ another regular cardinal,
    any $\kappa$-sumpact object is $\lambda$-sumpact.
\end{remark}

\begin{example}
    If the functor $\map_{\cC}(X,-)$ preserves coproducts,
    then $X$ is sumpact in $\cC$.
    This is often (e.g.~if $\cC$ is an $\infty$-topos) equivalent to being connected in the sense of being non-initial and admitting no non-trivial coproduct decomposition.
\end{example}

\begin{example}
    Let $\Ab^{\text{tf}} \subseteq \Ab$
    be the reflective subcategory  of torsion-free abelian groups.
    Then the Hom functor $\Hom(\Q,-) \colon \Ab^\text{tf} \to \Ab$
    sends a torsion-free abelian group to its subgroup
    of divisible elements.
    This functor preserves coproducts\footnote{Note though that the set-valued $\Hom(\Q,-) \colon \Ab^\text{tf} \to \Set$ does not preserve coproducts.},
    which implies that $\Q \in \Ab^\text{tf}$ is sumpact.
    On the other hand, the expression
    $\Q \cong \colim_n (\Z \xto{2} \Z \xto{3} \Z \xto{4} \cdots)$
    implies that $\Q$ is not compact in $\Ab^\text{tf}$.
\end{example}

\begin{remark}\label{rem:triangulated-terminology}
    In the theory of triangulated categories,
    an object is called $\kappa$-small if it is $\kappa$-sumpact in our terminology (\cite[Definition 4.1.1]{Neeman}).
    In the case $\kappa = \aleph_0$,
    the $\aleph_0$-small objects are also called \textit{compact}
    (\cite[Definition 4.2.7]{Neeman}).
    This is consistent with our terminology insofar as,
    in a cocomplete stable $\infty$-category $\cC$,
    an object $X$ is compact precisely if it is sumpact
    precisely if the mapping spectrum functor $\Hom_{\cC}(X, -) \colon \cC \rightarrow \Sp$ preserves colimits (it always preserves finite colimits by stability).

    For uncountable regular cardinals $\kappa$,
    there is also a notion of $\kappa$-compact object
    in a triangulated category (\cite[Definition 4.2.7]{Neeman}).
    For a stable $\infty$-category $\cC$ and an object $X \in \cC$,
    we do not know how $X$ being $\kappa$-compact in $h\cC$ in this sense
    relates to $X$ being $\kappa$-compact in $\cC$.
    If, however, $\cC$ is furthermore $\kappa$-presentable,
    it is possible to show the triangulated category $h\cC$ is $\kappa$-perfectly generated (\cite[Definition 8.1.4]{Neeman}),
    and the two conditions on $X$ are equivalent.
\end{remark}

\begin{remark}\label{rem:stable-cosumpact}
    Let $\cC$ be a presentable stable $\infty$-category.
    Then, by Theorem \ref{introthm:main} and the previous remark,
    any cosumpact object in $\cC$ is zero.
    We do not know whether also, for uncountable regular cardinals $\kappa$,
    the $\kappa$-cosumpact objects in $\cC$ are zero.
    In the case where $\cC$ is compactly generated,
    there is a closely related result of Neeman.
    Concretely, \cite[Appendix E.1]{Neeman}
    (translated from the triangulated to the stable $\infty$-categorical setting)
    shows that if $\cC$ is a non-trivial compactly generated stable $\infty$-category,
    then there exists an object $\mathbb{BC} \in \cC$
    (an analogue of the Brown--Comenetz dual spectrum $I_{\Q/\Z} \in \Sp$
    defined via Brown representability),
    which is not $\kappa$-cosumpact (cf.~Remark \ref{rem:triangulated-terminology} for the terminology) for any regular cardinal $\kappa$.
    In particular, he uses this to prove that non-trivial
    compactly generated stable $\infty$-categories are not copresentable,
    which is also an instance of Corollary \ref{introcor:bipresentable}.
\end{remark}

\subsection{Infinitary distributive categories}\label{subsec:distr}

In this subsection, we prove that for infinitary distributive
$\infty$-categories $\cC$, Theorem \ref{introthm:main} extends
to $\kappa$-cosumpact objects.
First, we show an auxiliary lemma.

\begin{lemma}\label{lem:factor-through-both-proj}
    Let $\cC$ be an $\infty$-category with finite products and $X \in \cC$.
    If there is a map $f \colon Y \times Z \to X$
    which admits a section and factors through both projections $\pr_Y$ and $\pr_Z$,
    then $X$ is $(-1)$-truncated.
\end{lemma}
\begin{proof}
    We have a commutative diagram
    \[\begin{tikzcd}
        X & {Y \times Z} & Z \\
        & Y & X,
        \arrow["s", from=1-1, to=1-2]
        \arrow["{\pr_Z}", from=1-2, to=1-3]
        \arrow["{\pr_Y}"', from=1-2, to=2-2]
        \arrow["f"{description}, from=1-2, to=2-3]
        \arrow["\phi", from=1-3, to=2-3]
        \arrow["\psi"', from=2-2, to=2-3]
    \end{tikzcd}\]
    where $fs \simeq \id_X$.
    By definition, we need to show that the diagonal $\Delta \colon X \to X \times X$
    is an equivalence.
    Due to $\Delta f \simeq \psi \times \phi \colon Y \times Z \to X \times X$,
    we observe that the following composite is inverse to $\Delta$:
    \[
        X \times X \xto{\pr_Y s\,\times\,\pr_Z s} Y \times Z \xto{f} X.\qedhere
    \]
\end{proof}

\begin{proposition}\label{prop:distr-case}
    Let $\cC$ be a bicomplete $\infty$-category and $\kappa$ a regular cardinal.
    Suppose that $\cC$ is \textnormal{$\kappa^+$-distributive}, i.e.~for every $X \in \cC$,
    the functor $X \times -$ preserves coproducts of size $\leq \kappa$ (e.g.~if $\cC$ is cartesian closed).
    Then every $\kappa$-cosumpact object in $\cC$ is $(-1)$-truncated.
\end{proposition}
\begin{proof}~
    Let $Y \in \cC$ be $\kappa$-cosumpact, $* \in \cC$ be terminal
    and $C = \coprod_\kappa *$.
    We have a commutative diagram
    \[\begin{tikzcd}[column sep=scriptsize]
        {\colim_{\alpha < \kappa}\map_\cC(C \times \prod_\alpha Y, Y)} & {\colim_{\alpha < \kappa}\map_\cC(\coprod_\kappa \prod_\alpha Y, Y)} & {\colim_{\alpha < \kappa}\prod_\kappa\map_\cC(\prod_\alpha Y,Y)} \\
        {\map_\cC(C \times \prod_\kappa Y, Y)} & {\map_\cC(\coprod_\kappa \prod_\kappa Y,Y)} & {\prod_\kappa\map_\cC(\prod_\kappa Y,Y),}
        \arrow["\simeq", from=1-1, to=1-2]
        \arrow["\simeq"', from=1-1, to=2-1]
        \arrow["\simeq", from=1-2, to=1-3]
        \arrow[from=1-2, to=2-2]
        \arrow[from=1-3, to=2-3]
        \arrow["\simeq", from=2-1, to=2-2]
        \arrow["\simeq", from=2-2, to=2-3]
    \end{tikzcd}\]
    where the two left horizontal equivalences use the assumption on $\cC$,
    and the left vertical equivalence uses $\kappa$-cosumpactness of $Y$.
    By 2-out-of-3, also the right vertical map is an equivalence.
    But the tuple of projections $(\pr_\beta \colon \prod_\kappa Y \to Y)_{\beta< \kappa}$
    is a point in $\prod_\kappa \map_\cC(\prod_\kappa Y,Y)$,
    which now must lie in the image of the right vertical map.
    Hence there exists some $\alpha < \kappa$
    for which every $\pr_\beta$ factors through $\pr_{< \alpha} \colon \prod_\kappa Y \to \prod_\alpha Y$.
    For $\alpha < \beta < \kappa$,
    the map $\pr_\beta \colon \prod_\alpha Y \times \prod_{\kappa \setminus \alpha}Y \simeq \prod_\kappa Y \to Y$
    admits a section (the diagonal $\Delta$) and factors through both projections.
    Hence $Y$ is $(-1)$-truncated by Lemma \ref{lem:factor-through-both-proj}.
 \end{proof}

\subsection{The relation to measurable cardinals}\label{subsec:measurable}

Recall that a cardinal $\kappa$ is measurable if it is uncountable
and admits a $\kappa$-complete non-principal ultrafilter,
see e.g.~\cite[Chapter 4.2]{Chang-Keisler}.
Measurable cardinals are inaccessible,
and thus their existence cannot be proven in ZFC.
Nevertheless, assuming their existence is not too uncommon (ask your local set theorist).
The reason measurable cardinals are relevant for us
is that, analogously to the case of cocompact objects
where we need to construct maps out of cofiltered limits,
to understand cosumpact objects we would like to construct
maps out of (infinite) products.
Of course, we always have the projections,
but in general it is hard to construct any other maps,
which is where ultrafilters come in by means of ultraproducts.
Following \cite[Chapter 4]{Chang-Keisler} or \cite[Section 3]{Chromatic}, we make the following definition.

\begin{definition}
    Let $\cC$ be an $\infty$-category with filtered colimits
    and small products.
    Let $I$ be some set and $\cU$ be an ultrafilter on $I$.
    The \textit{$\cU$-ultraproduct} of an object $X \in \cC$ is defined as
    \[
        (\prod_I X)/\cU \coloneqq \colim_{U \in \cU^\op} \prod_{i \in U}X,
    \]
    where the filtered colimit is taken over the diagram
    which for an inclusion $U \subseteq V$ of $\cU$-large subsets of $I$
    has as transition map the projection
    $\pr^V_U \colon \prod_V X \to \prod_U X.$
    In particular, there is a canonical map
    \[
        q \colon \prod_I X \to (\prod_I X)/\cU
    \]
    and this yields a diagonal map $q\Delta\colon X \to (\prod_I X)/\cU$.
\end{definition}

\begin{remark}\label{rem:set-ultraproduct}
    Classically, in $\cC = \Set$, we can describe the $\cU$-ultraproduct
    of a set $X$ as the quotient of $\prod_I X$ by the equivalence relation
    where $(x_i)_i \sim (y_i)_i$ precisely if $\{i \in I \mid x_i = y_i\} \in \cU$.
\end{remark}

\begin{remark}\label{rmk:proj-fact}
    Essentially by definition, the quotient map $q$
    factors through the projection
    $\pr_U \colon \prod_I X \to \prod_U X$
    whenever $U \subseteq I$ is $\cU$-large.
    In particular, if $\cU$ is the principal ultrafilter at $i \in I$, then $q = \pr_i$.
    In $\Set$, due to the concrete description given in Remark \ref{rem:set-ultraproduct}, we have a partial converse:
    if $|X| \geq 2$ and $q$ factors through $\pr_U$,
    then $U \in \cU$.
\end{remark}

The point of this is that we can now leverage the existence
of non-principal ultrafilters for the quotient maps $q$
to provide us with possibly interesting maps out of products.
For their concreteness, we will be especially interested
in ultraproducts of fields.

\begin{example}\label{ex:specprodfield}
    Let $K$ be a field and $I$ some set.
    Recall that the ultrapower of a field is calculated
    in $\Set$ and is again a field, and that there is a bijection
    \[
        \{\text{ultrafilters on $I$}\} \leftrightarrow \{\text{prime ideals in $\textstyle\prod_I K$}\},
    \]
    where $\cU \mapsto \frp_\cU \coloneqq \Ker(\prod_I K \to (\prod_I K)/\cU)$
    and in the other direction $\frp \mapsto \cU_\frp \coloneqq \{J \subseteq I \mid 1-\mathbbm{1}_J \in \frp\}$.
\end{example}

\begin{lemma}[{{\cite[Proposition 4.2.4]{Chang-Keisler}}}]\label{lem:diag-complete}
    Let $X$ be a set and $\cU$ an ultrafilter on some set $I$.
    The diagonal map $q\Delta \colon X \to (\prod_I X)/\cU$ is bijective if and only if $\cU$ is $|X|^+$-complete, i.e.~$\cU$ is closed under intersections of size $\leq |X|$.\qed
\end{lemma}

\begin{corollary}\label{cor:diag-equivalence}
    Let $\kappa$ be an uncountable regular cardinal.
    If $X \in \An^\kappa$
    and $\cU$ is a $\kappa$-complete non-principal ultrafilter on a set $I$,
    then the diagonal $X \to (\prod_{I}X)/\cU = \colim_{U \in \cU^{\op}} \prod_U X$ is an equivalence.
\end{corollary}
\begin{proof}
    Since $\kappa$ is uncountable,
    an anima $X$ is $\kappa$-compact if and only if $\pi_0X$
    and all its homotopy groups at all base points are $\kappa$-small.
    As homotopy groups commute with filtered colimits and products,
    we can reduce to the analogous statement for $X \in \Set^\kappa$,
    i.e.~to Lemma \ref{lem:diag-complete}.
\end{proof}

\begin{lemma}[{{\cite[Proposition 4.2.7]{Chang-Keisler}}}]
    \label{lem:smallest-meas}
    Let $\kappa$ be an uncountable regular cardinal.
    If there is a set $I$ admitting a $\kappa$-complete
    non-principal ultrafilter,
    then there exists a measurable cardinal $\mu \geq \kappa$.\qed
\end{lemma}

Recall that we denote by $\CAlg(K) = \CRing_{K/}$ the ordinary category
of $K$-algebras.

\begin{lemma}\label{lemma:fieldprod}
	Let $K$ be a field. There exists a morphism of $K$-algebras $\prod_IK\to K$ which does not factor through a projection $\pr_i$ if and only if $I$ admits a non-principal $|K|^+$-complete ultrafilter.
\end{lemma}
\begin{proof}
    Note first that a $K$-algebra map $f \colon \prod_I K \to K$
    is necessarily surjective, with the diagonal $\Delta \colon K \to \prod_I K$ as a section.
    It follows by Example \ref{ex:specprodfield} and Lemma \ref{lem:diag-complete}
    that the datum of such a map $f$ is equivalent to the datum of a $|K|^+$-complete ultrafilter $\cU$ on $I$.
    Then, Remark \ref{rmk:proj-fact} and the observation
    that $|K| \geq 2$ imply that $f$ factors through a projection
    if and only if $\cU$ is principal.
%
\end{proof}

\begin{proposition}\label{prop:field-measurable}
    Let $K$ be an infinite field. If $K \in \CAlg(K)$ is not cosumpact,
    then there exists a measurable cardinal $\mu \geq |K|^+$.
\end{proposition}
\begin{proof}
    Since the comparison map in Definition \ref{def:cosumpact}
    is injective, the assumption implies that it fails to be surjective.
    Thus there exists a set $I$ together with $K$-algebras $R_i,\, i \in I$
    and a map $f \colon \prod_{i \in I} R_i \to K$ which does not factor through any projection $\pr_{I_0}$
    onto a finite subset $I_0 \subseteq I$.
    Note that $f$ factors through $\pr_{I_0}$
    if and only if $f(\unit_{I_0}) = 1$, where $\unit_{I_0} \in \prod_{i\in I} R_i$ is the indicator function of $I_0$.
    This condition only depends on the composite $\prod_I K \xto{u} \prod_{i \in I} R_i \to K$,
    so we can reduce to the case $R_i = K$ for all $i \in I$.
    Therefore, by Lemma \ref{lemma:fieldprod}, $I$ admits a non-principal $|K|^+$-complete ultrafilter,
    and we conclude via Lemma \ref{lem:smallest-meas}.
\end{proof}

\begin{remark}
    This proof only ever used that $f$ did not factor through a projection
    onto a single factor, but this is equivalent to the assumption on $f$:
    any $K$-algebra map from a finite product to a field extension of $K$
    automatically factors through a projection onto a single factor.
\end{remark}

\begin{lemma}\label{lem:converse}
    Let $\kappa$ be a regular cardinal and $K$ a field.
    If $K \in \CAlg(K)$ is $\kappa$-cosumpact, then there exists no measurable cardinal of size $\geq \max(\kappa,|K|^+)$.
\end{lemma}
\begin{proof}
	Suppose for the sake of contradiction that $\mu \geq \max(\kappa, |K|^+)$ is measurable, and let $\cU$ be a non-principal $\mu$- and hence $|K|^+$-complete ultrafilter on $\mu$.
    By Lemma \ref{lemma:fieldprod}, we get a map of $K$-algebras $f \colon \prod_\mu K\to \prod_\mu K/\cU \cong K$ that does not factor through a single projection.
    As $K$ is $\kappa$-cosumpact, $f$ must factor through a projection onto some $\kappa$-small subset
    $J \subseteq \mu$. By Remark \ref{rmk:proj-fact}, this means that $J \in \cU$.
  Since $\cU$ is $\mu$-complete and $|J| < \kappa \leq \mu$, it follows that $\cU$ is principal, a contradiction.
\end{proof}

\begin{corollary}
    For an infinite field $K$, the following are equivalent:
    \begin{enumerate}
        \item $K$ is cosumpact in $\CAlg(K)$.
        \item There does not exist a measurable cardinal $\mu \geq |K|^+$.\qed
    \end{enumerate}
\end{corollary}

To complete the proof of Theorem \ref{introthm:measurable},
we still need the following lemma.

\begin{lemma}\label{lem:measurable}
    Let $\kappa,\lambda$ be regular cardinals.
    Suppose that $\cC$ is a $\lambda$-presentable $\infty$-category
    and $X \in \cC$ is $\kappa$-cosumpact.
    If there exists a measurable cardinal $\mu \ge \max(\kappa,\lambda)$
    such that $\map(Y,X) \in \An^\mu$ for all $Y \in \cC^\lambda$, then $X$ is $(-1)$-truncated.
\end{lemma}
\begin{proof}
    Let $\cU$ be a $\mu$-complete non-principal ultrafilter on $\mu$.
    In particular, $\cU^\op$ is $\mu$-filtered.
    By testing on $\map(Y,-)$ for all $Y \in \cC^\lambda$,
    we deduce from Corollary \ref{cor:diag-equivalence}
    that the diagonal map $q\Delta \colon X \to (\prod_\mu X)/\cU$
    is an equivalence.
    By $\kappa$-cosumpactness of $X$, there is a $\kappa$-small
    subset $I \subseteq \mu$ and a factorization
    \[\begin{tikzcd}
        {\prod_\mu X} & {(\prod_\mu X)/\cU} \\
        {\prod_{I} X} & {X,}
        \arrow["q", from=1-1, to=1-2]
        \arrow["\pr"', from=1-1, to=2-1]
        \arrow["f"{description}, from=1-1, to=2-2]
        \arrow["\simeq"', from=1-2, to=2-2]
        \arrow["{(q\Delta)^{-1}}", draw=none, from=1-2, to=2-2]
        \arrow[dashed, from=2-1, to=2-2]
    \end{tikzcd}\]
    where $f$ is defined as the composite.
    But since $I$ is $\kappa$- and hence $\mu$-small,
    the fact that $\cU$ is $\mu$-complete and non-principal
    shows that $I \notin \cU$ and, since $\cU$ is an ultrafilter,
    this yields that $\mu \setminus I \in \cU$.
    Thus $q$ and hence $f$ factor through $\prod_\mu X \to \prod_{\mu \setminus I}X$.
    In total, $f\Delta \simeq \id_X$ and $f$ factors through two complementary projections,
    so Lemma \ref{lem:factor-through-both-proj} implies that $X$ is $(-1)$-truncated.
\end{proof}

\begin{theorem}\label{thm:meas-card-eqv}
    The following statements are equivalent:
    \begin{enumerate}
        \item There exists a proper class of measurable cardinals.

        \item (Strengthened Theorem \ref{introthm:main})
            For any presentable $\infty$-category $\cC$ and regular cardinal $\kappa$,
            every $\kappa$-cosumpact object in $\cC$ is $(-1)$-truncated.

        \item For any presentable $\infty$-category $\cC$,
            every cosumpact object in $\cC$ is $(-1)$-truncated.

        \item For any infinite field $K$, the object $K \in \CAlg(K)$ is not cosumpact.
    \end{enumerate}
\end{theorem}
\begin{proof}
    Note that if $K$ is a field, then $K$ is never $(-1)$-truncated
    in $\CAlg(K)$, since $(-1)$-truncated objects are subobjects of the terminal one.
    Hence (2) $\Rightarrow$ (3) $\Rightarrow$ (4) is clear.
    Now (4) $\Rightarrow$ (1) by Proposition \ref{prop:field-measurable},
    so it remains to argue that (1) $\Rightarrow$ (2).

    So suppose that there is a proper class of measurable cardinals.
    Let $X \in \cC$ be a $\kappa$-cosumpact object for some regular cardinal $\kappa$
    in a presentable $\infty$-category $\cC$ and pick a regular cardinal $\lambda$
    such that $\cC$ is $\lambda$-presentable.
    Since $\cC^\lambda$ is small
    and $\cC$ is locally small,
    we can find a regular cardinal $\mu'$ so that $\map(Y,X) \in \An^{\mu'}$
    for all $Y \in \cC^\lambda$.
    Picking a measurable cardinal $\mu \ge \max(\mu',\kappa,\lambda)$,
    we deduce from Lemma \ref{lem:measurable}
    that $X$ is $(-1)$-truncated, as desired.
\end{proof}

\begin{remark}\label{rem:slender}
    Theorem \ref{thm:meas-card-eqv} has a partial counterpart
    in the theory of abelian groups.
    Namely, a theorem of Łoś--Eda (cf.~\cite[Theorem 13.2.8]{Fuchs}, \cite[Theorem III.3.2]{Eklof-Mekler})
    implies that, if there exist no measurable cardinals,
    \textit{slender} abelian groups are cosumpact in $\Ab$.
    The class of slender abelian groups is non-trivial
    and includes e.g.~the integers $\Z$.
    In fact, an argument analogous to Lemma \ref{lem:converse}
    also proves that if $\Z \in \Ab$ is cosumpact,
    then there exist no measurable cardinals.

    Note also the stark contrast between the derived and non-derived
    settings; in $\cD(\Z)$, Remark \ref{rem:stable-cosumpact}
    shows that $\Z \in \cD(\Z)$ is not cosumpact,
    whereas the non-existence of measurable cardinals
    would imply that $\Z$ is cosumpact in $\Ab$.
\end{remark}

\begin{remark}\label{rem:chaus}
    There is another partial counterpart to Theorem \ref{thm:meas-card-eqv} in general topology.
    Namely, consider a pointed path-connected compact Hausdorff space $X$ as an object in $\CHaus_{\ast}$.
    If there exist no measurable cardinals, the argument of \cite[Theorem 3]{Keesling-Rudyak} shows that,
    for any set $Y_i,\,i \in I$ of pointed compact Hausdorff spaces,
    the wedge $\bigvee_{i\in I} Y_i$ is a union of path-components in its Stone--Čech compactification,
    which is their coproduct in $\CHaus_{\ast}$.
    Thus any pointed map $X \rightarrow \beta(\bigvee_{i \in I} Y_i)$ factors through
    $\bigvee_{i \in I} Y_i$ and then through a finite sub-wedge (which is compact and thus the coproduct in $\CHaus_{\ast}$)
    by an easy general topology argument.
    Thus, $X$ is sumpact in $\CHaus_{\ast}$.

    In particular, if there exist no measurable cardinals, $\CHaus_{\ast}^{\op}$ is a presentable category
    with a proper class of non-$(-1)$-truncated cosumpact objects.
\end{remark}

\begin{remark}
    The proof of Theorem \ref{thm:meas-card-eqv} ultimately works by considering filtered colimits of products, while our proof of Theorem \ref{thm:main} uses cofiltered limits of coproducts. In some sense, this makes these two proofs ``Eckmann--Hilton dual'' to one another.
    This is even more apparent in the semiadditive setting: as mentioned in Remark \ref{rem:semiadditive-tail}, the fiber of $\bigoplus_{\aleph_0} X\to \prod_{\aleph_0}X $ is exactly $\lim_n \bigoplus_{\aleph_0 \setminus n}X$, which is the $\aleph_0$-tail of Proposition \ref{prop:1-tail}, while its cofiber (the ``Eckmann--Hilton dual'') is $\colim_n \prod_{\aleph_0\setminus n} X$, which is analogous to ultraproducts $(\prod_I X)/\cU = \colim_{U \in \cU^\op}\prod_{i \in U}X$ that have appeared in this section\footnote{Note also that the folklore argument that a compactly generated abelian category satisfying (AB5) and (AB5*) is zero proceeds by showing that $\bigoplus_{\aleph_0} X\to \prod_{\aleph_0}X$ is an isomorphism.}.

    Theorem \ref{thm:meas-card-eqv} shows that the product-type part of the cocompactness hypothesis alone does not suffice for an unconditional proof of Theorem \ref{thm:main}: its sufficiency for all presentable $\infty$-categories is equivalent to the existence of a proper class of measurable cardinals.
\end{remark}

\bibliography{reference}
\end{document}